\documentclass{amsart}
\usepackage{amsmath,amscd,amssymb,amsthm,epsfig,psfrag,latexsym,subfigure}

\newtheorem{theorem}{Theorem}[section]

\newtheorem{proposition}[theorem]{Proposition}

\newcommand{\ZZ}{\mathbb{Z}}

\newcommand{\CC}{\mathbb{C}}
\newcommand{\RR}{\mathbb{R}}

\def\vol{\mbox{\rm{Vol}}}

\def\area{\mbox{\rm{Area}}}

\def\d{\partial}
\def\D{\Delta}

\def\g{\gamma}

\def\G{\Gamma}

\def\S{\Sigma}
\def\split{\backslash\backslash}

\def\tr{\mbox{\rm{tr}}}

\def\PSL{\mbox{\rm{PSL}}}

\def\SL{\mbox{\rm{SL}}}

\def\Int{\mbox{int}}

\def\HH{\mathbb{H}}

\def\CC{\mathbb{C}}

\edef\t@mp{\catcode`\noexpand\#=\the\catcode`\#}%
    \catcode`\#=12
    \def\h@sh{#}%
    \t@mp

\edef\t@mp{\catcode`\noexpand\~=\the\catcode`\~}%
    \catcode`\~=12
    \def\tild@{~}%
    \t@mp

\begin{document}
\Large
\title{Pants immersed in hypebolic 3-manifolds}
\author{Ian Agol}
\address{
University of California, Berkeley \\
    970 Evans Hall \#3840 \\
    Berkeley, CA 94720-3840}

\email{ianagol@math.berkeley.edu}
\thanks{partially supported by NSF grant DMS-0504975 and the Guggenheim foundation}

\begin{abstract}
We show that an immersed thrice-punctured sphere in a cusped orientable hyperbolic
3-manifold is either embedded or has a
single clasp in a manifold obtained by hyperbolic Dehn filling on a cusp of the Whitehead link complement.
\end{abstract}

\maketitle

\section{Introduction}
Let $\S_{(0,3)}$ be a thrice-punctured sphere. 
Adams showed that if $M$ is an oriented 3-manifold with boundary
with hyperbolic interior such that $\S_{(0,3)}$ embeds incompressibly into $M$,
then $\Int M$ has an embedded totally geodesic thrice-punctured sphere \cite{Adams85}. 
Adams also gave many examples of hyperbolic manifolds containing
embedded incompressible thrice-punctured spheres, and indeed one may
easily produce infinite families of examples. 
We study immersed $\pi_{1}$-injective maps $f:\S_{(0,3)}\to M$, where $\Int M$
is hyperbolic and $f$ is not homotopic to  an embedding.  If one takes one component
of the Whitehead link, then it bounds an immersed 2-punctured disk in the complement of the other component, which 
has a single clasp (double-point arc) singularity as in Figure \ref{clasp}. One may perform surgery on the other component to obtain infinitely many
hyperbolic 3-manifolds with an immersed pants. We prove that
this is the only possible way that a pants may be non-trivially
immersed in a hyperbolic 3-manifold.

{\bf Notation:} For a path metric space $X$ and a closed
subspace $Y \subset X$, let $X\split Y$ denote the
path metric completion of the open subspace $X - Y$. 

{\bf Acknowledgement:} We thank Shawn Rafalski for 
suggesting some simplifications to the argument. 

\section{Parabolic $\PSL_2(\CC)$ representations of $\pi_1(\S_{(0,3)})$}

We will use $S=\S_{(0,3)}$ from now on to indicate a pair
of pants, that is a surface of genus zero with three boundary components. 
Let $\d S = c_1\cup c_2 \cup c_3$, and let $a_{ij}\subset S$
be an embedded arc connecting $c_i$ to $c_j$ (see Figure \ref{seams}). The arcs $a_{ij}$
are pairwise disjoint when $i\neq j$, and we may assume that there is a 
complete hyperbolic metric on $\Int S$ so that $\Int a_{ij}$ is
totally geodesic.  
\begin{figure}[htb] 
	\begin{center}
	\psfrag{a}{$a_{12}$}
	\psfrag{b}{$a_{23}$}
	\psfrag{c}{$a_{13}$}
	\psfrag{d}{$c_1$}
	\psfrag{e}{$c_2$}
	\psfrag{f}{$c_3$}
	\epsfig{file=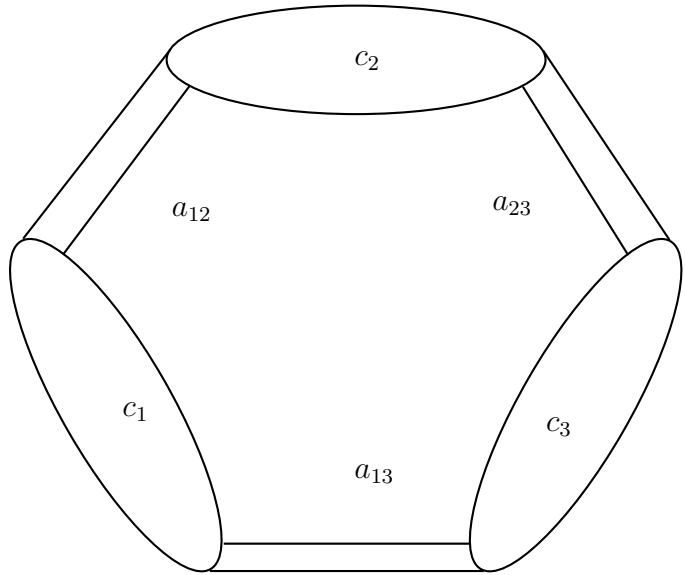,  height=.5\textwidth}
	\caption{\label{seams} Labelling a pants and three seams}
	\end{center}
\end{figure}
By taking a basepoint $x\in a_{12}$,
the map of the eyeglass graph $(c_1\cup c_2\cup a_{12}, x) \to (S,x)$ is a homotopy
equivalence. We may identify the generators of $\pi_1 (S,x)$
with $c_1$ and $c_2$, so that the third peripheral element
is $c_1 c_2$ corresponding to the boundary component $c_3$.  

There are two natural classes of representations of $\pi_1 S \to \PSL_2(\CC)$ for which the peripheral elements are parabolic.
The reducible representation is given by 
\begin{equation}
\rho(c_i)= \left(\begin{array}{cc}1 & z_i \\0 & 1\end{array}\right),
\end{equation}
where we have indicated matrix lifts of the generators
to $\SL_2(\CC)$. 
Another natural peripheral parabolic
representation is given by 
\begin{equation}
\rho(c_1)= \left(\begin{array}{cc}1 & 2 \\0 & 1\end{array}\right),\    
\rho(c_2)= \left(\begin{array}{cc}1 & 0 \\ -2 & 1\end{array}\right).
\end{equation}
The following proposition shows that up to conjugacy, these are the only peripheral parabolic representations of
$\pi_1 S$ (this is essentially due to Adams \cite{Adams85}).

\begin{proposition} \label{pprep}
Suppose we have a representation $\rho: \pi_1 S \to \PSL_2(\CC)$ such that the peripheral elements $\rho(c_1), \rho(c_2),
\rho(c_1 c_2)$ are all parabolic. Then either $\rho$ is reducible
and we may conjugate $\rho$ to Equation (1), or $\rho$ is
conjugate to Equation (2). 
\end{proposition}
\begin{proof}
If $\rho(c_1)$ and $\rho(c_2)$ fix the same point on $\partial \HH^3$,
then we may send this point to $\infty$, so that the representation
is conjugate to the reducible representation given above in equation (1). 
Otherwise, assume that $\rho(c_1)$ and $\rho(c_2)$ fix
different points on $\partial \HH^3$. By a further conjugation,
we may send the fixed point of $\rho(c_1)$ to $\infty$, and
the fixed point of $\rho(c_2)$ to $0$ so  that we may assume 
that 
\begin{equation}
\rho(c_1)= \left(\begin{array}{cc}1 & 2 \\0 & 1\end{array}\right), 
\rho(c_2)= \left(\begin{array}{cc}1 & 0 \\ z & 1\end{array}\right),
\rho(c_1 c_2) =  \left(\begin{array}{cc}1+2 z & 2 \\ z & 1\end{array}\right).
\end{equation}
Since $\rho(c_1 c_2)$ must be parabolic, we have 
$\tr(\rho(c_1 c_2) )= 2+2z = \pm 2$, so $z=0$ or $z=-2$.
The case $z=0$ means that $\rho(c_2)$ is trivial,
so that the representation is reducible, contrary to 
assumption. Thus, we must have $z=-2$, and
we have the representation given in Equation (2). 
\end{proof}

\section{Whitehead link complement} 

There is a well-known class of immersed pants in 3-manifolds
which comes from drilling out a knot bounding an immersed disk with a single clasp
singularity in a 3-manifold (see Figure \ref{clasp}). Taking a regular neighborhood
of the disk with a clasp gives a solid torus. Since the boundary
of the disk forms a knot inside of the
solid torus, if the drilled manifold is to be hyperbolic, the
torus bounding the solid torus must be compressible in the
complement of the knot,
and therefore bounds a solid torus on the outside as
well, or else bounds a torus $\times I$. We immediately 
see that the only way this may happen is that we have
Dehn filling on one component of the Whitehead link
complement, or the Whitehead link complement itself.

\begin{figure}[htb] 
	\begin{center}
	\epsfig{file=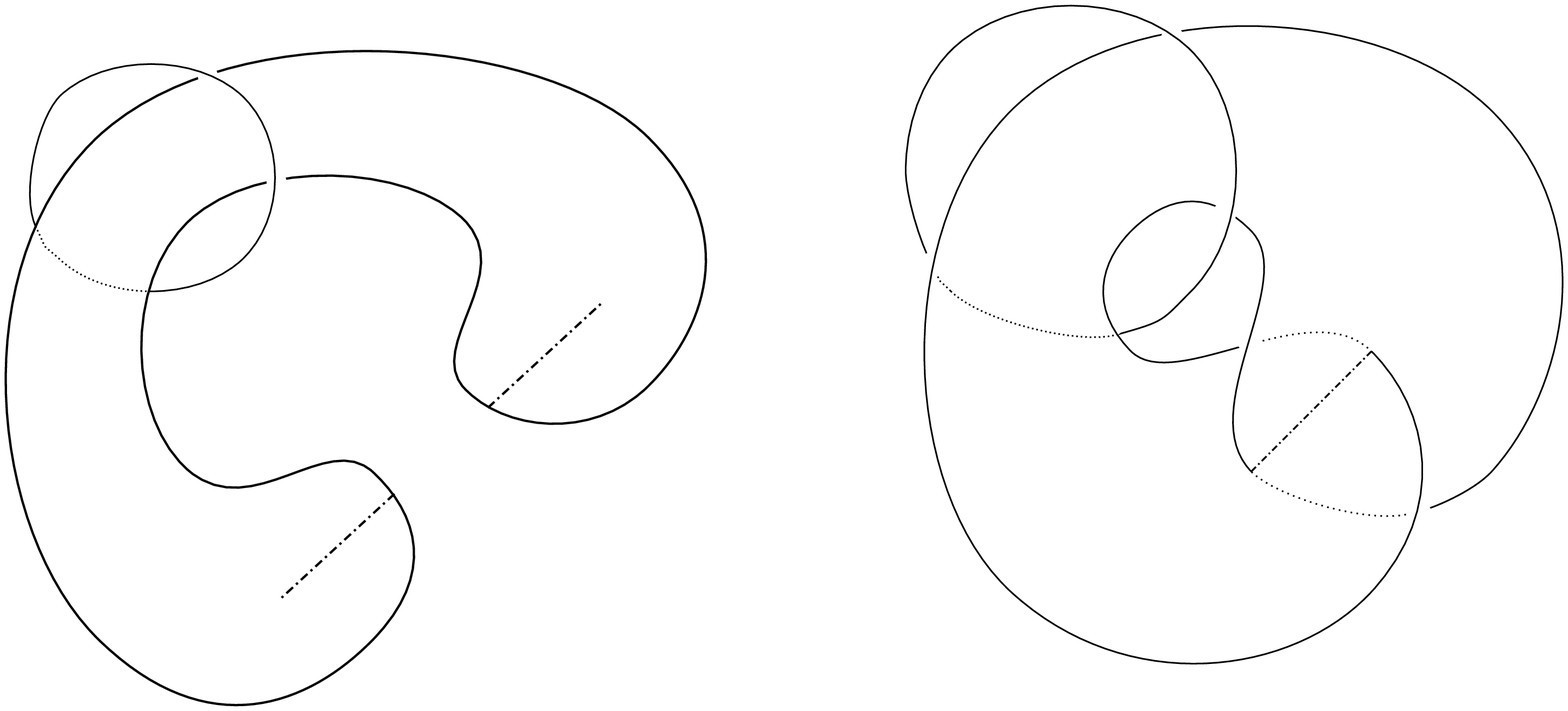, width=\textwidth, height=.5\textwidth}
	\caption{\label{clasp} A twice-punctured disk in the Whitehead
	link complement}
	\end{center}
\end{figure}

In fact, we may 
 use this immersed twice-punctured disk
bounding one component of the Whitehead link to parameterize
(generalized) hyperbolic Dehn filling on the other boundary component. Let $W\subset S^3$ be the Whitehead link, and
let $M=S^3 \split \mathcal{N}(W)$ be its complement. Let $\partial M=T_1\cup T_2$ be the two torus boundary components of $M$ corresponding to the two components of $W$. Let $S$ be a pants, and  $f: S \to M$ be 
the immersion of $S$ into $M$ such that $f(\partial S)\subset T_2$.
Suppose that we have a non-elementary representation $\rho: \pi_1 (M) \to \PSL_2(\CC)$,  such that  $\rho(\pi_1 T_2) < \PSL_2(\CC)$ is parabolic. 
Under the map $f:S\to M$ we may
assume that $f(a_{13})=f(a_{23})$ is the set of double points
of the immersion.  Let $q\in \pi_1(M)$ be an element
which sends $a_{23}$ to $a_{13}$, in such a way that
$\rho(q)$ sends the fixed point of $\rho(c_2)$ to the
fixed point of $\rho(c_1c_2)$, and the fixed point of 
$\rho(c_1c_2)$ to the fixed point of $\rho(c_1)$ (see Figure \ref{cover}).

\begin{figure}[htb] 
	\begin{center}
	\psfrag{a}{$a_{12}$}
	\psfrag{b}{$a_{23}$}
	\psfrag{c}{$a_{13}$}
	\psfrag{d}{$c_1$}
	\psfrag{q}{$q$}
	\psfrag{e}{$c_2$}
	\psfrag{h}{$\HH^3$}
	\epsfig{file=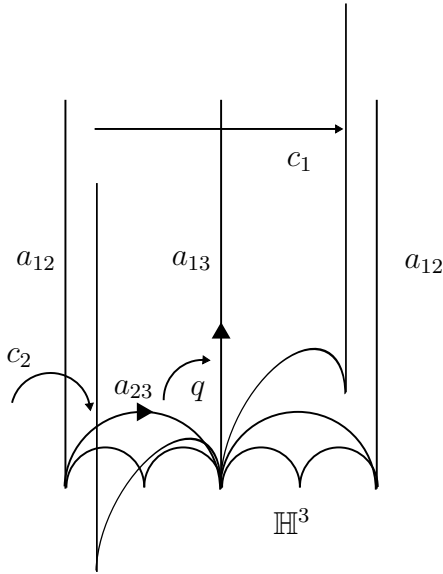, height=.5\textwidth}
	\caption{\label{cover} The universal cover of the immersion
	$f:S \to M$}
	\end{center}
\end{figure}
These satisfy the relations $c_1q^2 = q^2 c_2$ and $[q^{-1}c_1 q, c_1c_2] = 1$.
One may check that $\{ c_1, c_2, q\}$ generate $\pi_1(f(S))$,
and therefore $\pi_1(M,x)$,
because $M$ has a retraction onto the 2-complex $f(S)\cup T_2$,
since $M\split f(S) \cong T_1\times [0,1]$ (see the discussion
at the beginning of the section), and $\pi_1(f(S))$ generates $\pi_1(f(S)\cup T_2)$. 
Then $\rho(f_{\#}(\pi_1 S))$ will be a representation for which the
three peripheral elements of $\pi_1 S$ are parabolic. We may identify
$\rho(f_{\#}(\pi_1 S))=\langle C_1, C_2\rangle$, where 
$C_i=\rho(f_{\#}(c_i))$. 
There are two cases
depending on whether this representation is reducible or not. 

If $\langle C_1, C_2\rangle$ is reducible, we may assume that
$\rho$ is conjugated so that this subgroup looks like Equation (1).
Then $\rho(q)$ must also fix $\infty$ since $q^{-1}c_1 q$ and
$c_1 c_2$ commute and therefore both fix $\infty$. In this case we see
that $\rho$ is reducible. 

If $\langle C_1, C_2 \rangle$ is parabolic and irreducible, we
may assume $\rho$ is conjugated so that this subgroup looks
like Equation (2) by Proposition  \ref{pprep}. Then $a_{12}$ lifts to a geodesic arc in $\HH^3$
connecting $0$ and $\infty$, $a_{13}$ connects $\infty$ and $1$,
and $a_{23}$ connects $0$ and $1$ (see Figure \ref{cover}). Keeping track of orientations,
$\rho(q)(0)=1$, and $\rho(q)(1)=\infty$, which implies that
\begin{equation}
\rho(q) = \left(\begin{array}{cc}a^{-1}-a & a \\-a & a\end{array}\right).
\end{equation}
Thus, we see that irreducible representations of $\pi_1 M$
which are parabolic on one cusp are parameterized by a single
number $a\in \CC-\{0\}$. 

More importantly, the above discussion implies the following
intermediate result, which will be used in the proof of the main theorem:
\begin{proposition} \label{Whiteheadimage}
Let $N$ be an orientable compact 3-manifold, such that
$\Int N$ admits a complete hyperbolic metric of finite volume. 
Suppose that $f:(S,\d S) \to (N, \d N)$
is an essential map such that $f(a_{13})=f(a_{23})$. Then
$N$ is obtained by Dehn filling on one component of the
Whitehead link complement. 
\end{proposition}
\begin{proof}
Let $\G<\PSL_2(\CC)$ be the holonomy of $\pi_1 N$. 
Let $\langle C_1, C_2 \rangle < \PSL_2(\CC)$ correspond
to $f_{\#}(\pi_1 S)$, where we see that $C_i$ and $C_1C_2$ are
parabolic since $\Int N$ has finite volume. If $\langle C_1, C_2\rangle$ is reducible, then
$\pi_1 S$ isn't essential, so we may assume that $\langle C_1, C_2\rangle$ is normalized as in Equation (2) by Proposition \ref{pprep}. 
Let $Q$ be an element
of $\G$ sending a lift of $a_{13}$ to a lift of $a_{23}$, so 
that $Q$ is normalized as in Equation (4). Then we see
$\rho(\pi_1 M) < \G$, where $M$ is the Whitehead link
complement and $\rho$ is the representation defined by 
Equations (2) and (4). Moreover, $\rho$ is parabolic on
one boundary component  of $M$. Since $\G$ is
discrete, $\rho(\pi_1 M)= \G' <\G$ is a finite-index subgroup.
By \cite{FrancavigliaKlaff06}, $\vol(\HH^3/\G') \leq \vol(\Int M) = 3.66\ldots$. If $[\G:\G'] >1$, then $\vol(\Int N) =\vol(\HH^3/\G) < 1.84$,
which contradicts $\vol(\Int N) \geq 2.0298$ by \cite{CM}. 
Thus, we see that $\G'=\G$, so $\G$ is the discrete torsion-free homomorphic
image of $\pi_1 M$. In this case, we apply the analysis of \cite{NR92}
to see that every discrete irreducible homomorphic image of $\pi_1 M$ in $\PSL_2(\CC)$ with one cusp remaining parabolic must
come from a hyperbolic structure on the Whitehead
link complement with Dehn surgery type singularity along
the other boundary component. 
This follows from the parameterization of irreducible
representations $\rho(\pi_1 M)$
given in the proof of Theorem 6.2 of \cite{NR92} in
terms of a single parameter $z \in \CC-\{0\}$. 
Their parameter $z=x-x^{-1}$, where $x, -x^{-1}\in \CC-\{0,\pm 1\}$
are the complex parameters of two pairs of tetrahedra in an
ideal triangulation of $M$ corresponding to a hyperbolic
structure of $\Int M$ with $T_2$ parabolic and a
Dehn surgery type singularity along $T_1$. The
parameter $z$ must correspond to our parameter
$a\neq 0$ up to a M\"obius transformation, since any proper subset
of $\CC-\{0\}$ is not holomorphically equivalent to $\CC-\{0\}$.
Since $\G$ is discrete and torsion-free, we conclude that 
$\rho(\pi_1 T_1)$ must be a cyclic group. 
Therefore, the completion of $\Int M$ gives a cone-manifold structure $M'$ with
cone angle an integral multiple of $2\pi$ at the Dehn
surgery type singularity along $T_1$ (see \cite{CH03} for a discussion
of cone manifolds). If the 
cone angle is not $2\pi$, this
implies that the map $M' \to \HH^3/\G$ must be a branched 
immersion which is at
least two-to-one, which would imply that $\vol(\HH^3/\G)< 1.84$,
giving a contradiction as before. 
\end{proof}

\section{Immersed pants}

The goal of this section is to prove the following theorem
which shows that these examples give the only way
that a pants may be non-trivially immersed in a hyperbolic
3-manifold. 

\begin{theorem}
Let $M$ be an orientable 3-manifold with boundary and interior admitting a complete hyperbolic metric. Let $S$ be
a pants, and $f:(S,\partial S) \to (M,\partial M)$ an immersion. 
Then either $f$ is relatively homotopic to an embedding, or 
$M$ is obtained by (possibly empty) Dehn filling on 
one boundary component of the Whitehead link complement,
and $f$ may be homotoped to have a single clasp
singularity, {\it i.e.} a single embedded arc of double points as in Figure \ref{clasp}. 
\end{theorem}
\begin{proof}
We may assume that the map $f$ is homotoped so
that the restriction $f_{| \Int(S)}$ is totally geodesic with respect to the canonical metric on $\Int(M)$. We may identify $\Int M= \HH^3/\G$,
where $\pi_1 M\cong \G < \PSL_2(\CC)$ is a torsion-free discrete group. Throughout the argument, we will fix orientations on $S$
and $M$, as well as the induced orientations on $\d S=c_1\cup c_2\cup c_3$ and on $\d M$. The proof will proceed by deducing
a sequence of restrictions on the manifold $M$ and the 
nature of the immersion $f$. 

{\bf Claim:} {\it  If $f$ is
not an embedding, then $\Int M$ must have finite volume. }
This is achieved via an area estimate. Consider 
$N=\Int M \split  f(\Int S)$, which has a hyperbolic metric with
convex boundary. For each component $N_i $ of $N$,
there is a convex core $CH(N_i)$ (which we take to 
be empty if $\chi(N_i)\geq 0$). Each $\partial CH(N_i)$
has an intrinsic hyperbolic metric. If $\partial CH(N_i)$
has a cusp, then there is some neighborhood of the 
cusp which is totally geodesic inside of $N_i$. This cusp must be
parallel to a cusp of $\partial N_i$, which corresponds
to an embedded cusp of $f(\Int S)$. But this cusp of
$f(\Int S)$ has two sides, and thus there is a component
$N_j$ on the other side of this cusp (it's possible that $j=i$),
with a cusp of $\partial N_j$, and therefore of $\partial CH(N_j)$ (which will be a distinct cusp of $\partial CH(N_i)$ if $i=j$).
Therefore, the cusps of $\partial CH(N)$ must come in pairs,
which implies that $\chi(\partial CH(N))$ must be even. 
Throw out all components $N_i$ of $N$ with $CH(N_i)=\emptyset$
to obtain $N'\subset N$ with $\chi(N')<0$ if $N'\neq \emptyset$. 
Then $\chi(\partial CH(N'))<0$, and therefore $\chi(\partial CH(N'))\leq -2$,
since it is even. 

There is a nearest point 1-Lipschitz retraction
$N'\to CH(N')$, which is area decreasing when restricted
to a map $\partial N' \to \partial CH(N')$ \cite{EM87}. We have $Area(\partial CH(N'))=-2\pi \chi(\partial CH(N'))\geq 4\pi$. But $4\pi = Area(\partial N) \geq Area(\partial CH(N'))$, which implies that $Area(\partial N')=Area(\partial CH(N))$. This can occur if and only if $N=CH(N')$. Since $\partial CH(N')$
is bent along a compact measured lamination and  $\partial N$ is piecewise
composed of pieces of $f(S)$, $\partial N$ must be bent along
the double points of $f(S)$. But then $\partial N$ must be bent along
 a simple closed geodesic. This is impossible, since there
are no simple closed geodesics in $S$. Thus, we have a contradiction,
and we conclude that $N'=\emptyset$. This implies that all
components of $\Int N$ are balls, tori, or $T^2\times \RR$.
This implies that $N$ has finite volume.

{\bf Claim:} {\it For each boundary component $c$ of $\d S$, $f(c)$ is embedded. }
Otherwise, since $f$ is totally geodesic,  $f$ would have degree at least two onto its image. Then $f$ would factor
through a covering  $f':S\to S'$ of degree $>1$, which is impossible since $\chi(S)=-1$ so $S$ does not cover
any surface $S'$ non-trivially.

{\bf Claim:}  {\it The image of the boundary $f(\d S)$ meets
at most two boundary components of $M$. }

Suppose that $f(\d S)$ meets three boundary components of $M$. Then each component of $\d S$
maps to a distinct boundary component of $M$.
Thus, for each component $c\subset \d S$, $0\neq f_{\ast}[c] \in H_{1}(\d M)$, and therefore $0\neq f_{\ast}([S]) \in H_{2}(M,\d M)$. Since $\Int(M)$ is
hyperbolic, each homology class of $H_{2}(M,\d M)$ has Thurston
norm $\geq 1$ \cite{Th3}. There exists an embedded orientable incompressible surface $\S\subset M$, $f_{\ast}([S])=[\S]\in H_{2}(M,\d M)$, by \cite{Ga2}, and such that $\chi_{-}(\S)=\chi_{-}(S) = 1$, which implies that $\S$ is connected and $\chi(\S)=-1$. Since 
each boundary component of $S$ goes to a different boundary
component of $M$ under the map $f$, we see that  $\S$ is an embedded pants such that $\Int \S$ is totally geodesic. Also,
each component of $\d \S$ is parallel to a component of $f(\d S)$, since
each component of $f(\d S)$ is homologous to a component of $\d \S$. 
Since $f$ and $\S$ are both totally geodesic, the boundary 
components  $f(\d S)\cap \d \S = \emptyset$, since otherwise
we would have $f(S)=\S$, which contradicts the assumption that
$f$ is not embedded. Then $f^{-1}(\S)$ is a collection of embedded
geodesic curves, since $\S$ is embedded. But these curves miss 
$\d S$, which means $f^{-1}(\S) = \emptyset$, since there are no
closed embedded geodesics in $\Int S$. But this is also a contradiction,
since $M\backslash f(S)$ is a union of regions with abelian
fundamental group, and
thus cannot contain the embedded pants $\S$. 

{\bf Claim:} {\it The image of the boundary $f(\d S)$ 
meets only one boundary component of $M$.} 

Suppose  that $f(\partial S)\subset T_{1}\cup T_{2} \subset \d M$, where
$T_{i}$ is a torus. Let $\d S = c_{1}\cup c_{2} \cup c_{3}$.
Then we may assume that $f(c_{1}\cup c_{2})\subset T_{1}$, and
$f(c_{3})\subset T_{2}$, since we are assuming for contradiction
that $f(\d S)$ is not contained in a single boundary component. 
This implies that $0\neq f_{\ast}([c_{3}])\in H_{1}(T_{2})$,
and thus $0\neq f_{\ast}([S]) \in H_{2}(M,\d M)$. Let $\S\subset M$
be an embedded orientable incompressible surface such that $[\S] = f_{\ast}([S])\in H_{2}(M,\d M)$, and $\chi_{-}(\S)=1$. Then $\S$ is 
either a pants, or a punctured torus. First assume that $\S$ is a pants. Then we may assume that $\Int(\S)$ is totally geodesic in $\Int(M)$.
Let $\d \S = c_{1}'\cup c_{2}'\cup c_{3}'$, where $c_{3}'\subset T_{2}$  is
a curve parallel to $f(c_{3})$, which must be disjoint since $f(S)$ and $\S$ do not coincide. 
Then $f^{-1}(\S)\subset S$ consists of embedded geodesics,
which must be disjoint from $c_{3}$.  Since $c_{1}'$ and $c_{2}'$ 
are parallel, and $[c_{1}']+[c_{2}']=f_{\ast}([c_{1}]+[c_{2}])\in H_{1}(T_1)$, either $f(c_{i})\cap \d \S =\emptyset$ for both $i=1,2$,
or $|f(c_{i})\cap \d \S|=2$. In the first case, we obtain a 
contradiction as before in the case that we assumed that $f(\partial S)$ meets three boundary components of $M$. In the second case, we have a contradiction, since
there is no embedded geodesic curve  $\g \subset S$ such that
$\partial \g\subset c_1\cup c_2$ and $|\g\cap c_i|=2$, $i=1,2$.

Thus, we must be in the case that $\S$ is a punctured torus. 
We may assume that $\S$ is quasifuchisan,
that is $\S$ has no accidental parabolics, otherwise there would be
an embedded pants homologous to $\S$, giving a contradiction as before. 
We may assume that $f(c_{3}) \cap \d \S = \emptyset$, 
by an isotopy, since these curves are homologically
parallel, and therefore isotopic in $\partial M$. We have $f^{-1}(\S)$ is an embedded union of curves
on $S$ which miss $\d S$. Homotope $f$ so that the number of
components of $f^{-1}(\S)$ has 
minimal cardinality. We may assume that each component
of $f^{-1}(\S)$ is essential in $\S$, since otherwise a homotopy
would reduce the cardinality of $|f^{-1}(\S)|$. Therefore if $f^{-1}(\S)$ is non-empty, 
each component
of $f^{-1}(\S)$ is an embedded closed curve which is
boundary parallel in $S$. An outermost such curve $c$ on $S$ cobounds
an annulus $A$ with a boundary component $c_{i}$ of $S$. 
Then $f(c)$ represents a parabolic element in $\S$, and therefore
must be boundary parallel in $\S$ since $\S$ has no accidental parabolics. 
Thus, since $M$ is acylindrical, we may homotope $f$ to reduce the number of components
of $f^{-1}(\S)$, a contradiction. If $f^{-1}(\S)$ is empty, we obtain 
a contradiction as before. Thus, we conclude that $f(\d S)$ 
meets at most one boundary component of $M$.

{\bf Claim:} {\it $f(\d S)$ is not embedded.} 
For sake of contradiction, assume that $f(\d S)\subset \d M$
is embedded. In this case, again we have 
$0\neq f_{\ast}([\d S])\in H_{1}(\d M)$, so we have 
$0\neq f_{\ast}([S]) \in H_{2}(M, \d M)$. Thus, there is
an embedded incompressible norm-minimizing surface 
$\S\subset M$, such that $\chi_{-}(\S)=1$, $f_{\ast}([S])= [\S]\in H_{2}(M,\d M)$. 
If $\S$ is a pants, then $\d \S$ must be parallel to $f(\d S)$,
since they are homologically parallel. We obtain a contradiction
as before in the case that $M$ had three boundary components. 
If $\S$ is a punctured torus with no accidental parabolic, then $f^{-1}(\S)\subset S$
is a collection of embedded closed curves, which are therefore
boundary parallel. As in the previous paragraph, we homotope
$f$ so that $f^{-1}(\S)$ has fewer components, until $f^{-1}(\S)$
is empty, which gives a contradiction as before, proving that
$f(\d S)$ is not embedded.

We may now assume that $f(\d S)$ meets only one boundary component of $M$, and that the curves $f(c_{i})$ are not all parallel
in $\d M$. We introduce some further notation now to be used
throughout the rest of the proof. 
Let $H \subset \Int M$ be a maximal open horocusp
containing neighborhoods of the ends of $f(\Int S)$ (thus homeomorphic to $T^2\times \RR$). 
Then $\overline{H}\subset \Int M$ will have self-tangencies. 
The preimage of $H$ in $\HH^{3}$ is a family of horoballs invariant
under the action of $\G\cong\pi_{1}(M)$. The preimage $f^{-1}(H)\subset \Int S$
contains a collection of horocycles in $\Int S$ surrounding $\d S$. There are three curves 
$h_{i} \to \Int S$,  such that $\cup_{i=1}^{3} h_{i} \subset f^{-1}(\d \overline{H})$ and 
such that $h_{i}$ is homotopic to $c_{i}$ in $S$. The curves $h_{i}$ might 
not be embedded, since they may have self-tangencies mapping to the 
self-tangent points of $\overline{H}$ under the map $f$. The curves
$f(h_{i})\subset \d\overline{H}$ are geodesics in the intrinsic euclidean
metric on $\d \overline{H}$.

{\bf Claim:} 
For $i\neq j$, we have $l(h_{i})l(h_{j})\leq 4$. 

This may be shown by 
a simple computation in hyperbolic geometry. One way to see this is
to expand $h_{i}$ and $h_{j}$ keeping them horocycles, until they 
become tangent horocycles $h_{i}', h_{j}'$. Then $l(h_{i}')l(h_{j}')=4$.
This may be shown by shrinking the longer of the two cycles, while
expanding the smaller, until we obtain two tangent horocycles with the same
length. This operation preserves the product of the lengths, since
if we shrink a distance $d$, then the length gets multiplied by $e^{-d}$,
while the expanded horocycle has length multiplied by $e^{d}$, keeping
the product constant. Once we reach equal size tangent horocycles,
both have length 2, so the product is 4.

If $T$ is a torus, let $\D: H_{1}(T) \times H_{1}(T) \to \ZZ$ be the algebraic intersection number. For embedded oriented curves $a,b \subset T$,
we will use the notation $\D(a,b)=\D([a],[b])$.

{\bf Claim:}  $\D(f(c_i),f(c_j))\leq 1$, $\forall i,j$. 

If $f(h_{i}), f(h_{j})$ are horocycles in $\d \overline{H}$, 
then $|\D(f(h_i), f(h_j))| \area(\d\overline{H})  \leq l(h_{i})\cdot l(h_{j})$. 
This implies that for any $i\neq j$, if $|\D(f(h_{i}),f(h_{j}))| >1 $,
then $\area(\d \overline{H}) \leq 2$. But this is a contradiction,
by \cite[Prop. 5.8]{CM} which
states that the area of $\d \overline{H}$ is $\geq 3.35$. 
Thus, we conclude that $|\D(f(c_i),f(c_j))|\leq 1$, $\forall i,j$.

{\bf Claim:} 
{\it We may assume that $0\neq f_{\ast}([\d S]) \in H_{1}(\d M)$. }

Suppose that $f([\d S]) =0 \in H_{1}(\d M)$, so that
$[f(c_{3})] = -[f(c_{1})]-[f(c_{2})]$ (with appropriately chosen orientations).  
Then 
$$\D( f(c_{1}),f(c_{2}) )+\D(f(c_{1}),f(c_{3})) + \D(f(c_{2}),f(c_{3}))
=$$
$$\D(f(c_{1}),f(c_{2})) + \D(f(c_{1})+f(c_{2}), -f(c_{1})-f(c_{2})) = \D(f(c_{1}),f(c_{2})),$$
since $\D$ is bilinear and skew-symmetric. But we also have that 
$|\D(f(c_{1}),f(c_{2}))| + |\D(f(c_{1}),f(c_{3}))|+|\D(f(c_{2}),f(c_{3}))|$ represents
the number of endpoints of the arcs of self-intersection of the map $f:S\to M$,
and therefore must be even. Therefore
$\D( f(c_{1}),f(c_{2}) )+\D(f(c_{1}),f(c_{3})) + \D(f(c_{2}),f(c_{3}))$ is
also even, since it has the same parity. 
 Thus, we see that $\D(f(c_{1}),f(c_{2}))$ must
be even, which implies that $\D(f(c_{1}),f(c_{3}))$, $\D(f(c_{2}),f(c_{3}))$ are
also even, by symmetry of the indices. Since $\D(f(c_i), f(c_j)) \neq 0$
for some $i, j$ because $f(\d S)$ is not embedded, we must have $|\D(f(c_i),f(c_j))|=2$, which contradicts the previous claim and thus
implies that $0\neq f_{\ast}([\d S])\in H_1(\d M)$. 

Since  $|\D(f(c_i),f(c_j))| \leq 1$, and the total parity of the
3 intersections is even, we conclude that $\D(f(c_i),f(c_j))=0$
for some $i\neq j$, that is  $f(c_{i})$ and $f(c_j)$ must be parallel (or
anti-parallel, keeping track of orientations). 
Since not all three are parallel, we may assume that $f(c_{1}), f(c_{2})$
are parallel, and $f(c_{3})$ intersects both precisely once. 
Thus, we have $l(f(h_1))=l(f(h_2))=b$, $l(f(h_3))=a$ for
some $a, b$ such that $ab\leq 4$.  Moreover, $1\leq b\leq 2$
where the lower bound follows from the fact that the length of a horocycle
in a maximal cusp is $\geq 1$ \cite{A02},  and the upper bound
follows from the facts that $l(h_1)=l(h_2)=b$ and $l(h_1)l(h_2)\leq 4$. 
Clearly we also have $a\leq 4$.  

In $S$, there are embedded essential arcs (unique up to isotopy) $a_{ij}$ connecting $c_{i} $ to $c_{j}$.
We may assume that $\Int a_{ij}\subset \Int S$ is a geodesic. 
We may use these arcs to analyze
the preimage $\tilde{H} \subset \HH^3$ of the horoball neighborhood of the cusp
$H\subset \Int M$. Identifying $\HH^{3}$
with the upper half space model, we may conjugate $\G$ so
that there is a component  of $\tilde{H}$ which is  
a horoball $H_{\infty} \subset \HH^{3}$ centered at 
$\infty$, so that the boundary of $H_{\infty}$ is
a Euclidean plane at height 1 and the intrinsic hyperbolic
metric on $\d H_{\infty}$ is the same as the induced
Euclidean metric. Then $\tilde{H} = \cup_{p\in \G(\infty)} H_p$,
where $H_p$ is a horoball component of $\tilde{H}$ centered at $p\in \d \HH^3$. Up to the stabilizer of $H_{\infty}$,
there are two lifts of each geodesic arc $\Int a_{ij}$ to $\HH^{3}$ with
one endpoint at $\infty$, corresponding to the
two ends of $a_{ij}$. If $f(a_{13})\neq f(a_{23})$, then
we see four horoballs from infinity of height $ab/4$ up
to the stabilizer of $H_{\infty}$ corresponding to the
four distinct lifts of these two arcs,
where $a=l(h_{3}), b=l(h_{1})$. By possibly conjugating
$\G$, we may assume that the four horoballs are centered
at $0, a/2, w_1, w_2$, where $w_1, w_2\in \CC$. 

{\bf Claim:} {\it If $f(a_{13})\neq f(a_{23})$, then  $ab<4$.}

If $f(a_{13})\neq f(a_{23})$ and  $ab=4$, then
these four distinct horoballs (up to the stabilizer of $\infty$) have height 1. We show that this gives
a contradiction. First, we will consider the case that $a<4$ and $b<2$. We get three strings of tangent horoballs up to the stabilizer
of $\infty$ corresponding to each cusp of $S$. One string consists of height 1
horoballs $H_{ka/2}$ and
horoballs of height $a^2/16$ $H_{ka/2+a/4}$, $k\in \ZZ$. The other pair of strings has height
1 horoballs $H_{ku+w_i}$, $|u|=b, u\in \CC-\RR$, $k\in \ZZ$, $w_1, w_2 \in \CC$ and horoballs of height $b^2/4$ $H_{(k+1/2)u+w_i}$. Then the horoballs $H_{ku+w_i}$
must each be disjoint from the horoballs $H_{ka/2}$ since $f(a_{13})\neq f(a_{23})$, and
disjoint from the horoballs $H_{ka/2+a/4}$ since they have distinct heights (since $a<4$ by assumption). This implies that
$|ku+w_i - k' a/2| \geq 1$ and $|ku+w_i - (k'a/2+a/4)|\geq a/4$ for
all $k, k'\in \ZZ$. 
A geometric computation shows that the only points
satisfying this property must
be distance $\sqrt{1-4/a^2}$ from the real axis (see Figure \ref{ab=4})
when $a\geq 2\sqrt{2}$, and distance $\sqrt{1-(a/4)^2}$ when $a\leq 2\sqrt{2}$.
\begin{figure}[htb] 
	\begin{center}
	\psfrag{a}{$a/4$}
	\psfrag{1}{$1$}
	\psfrag{h}{$\sqrt{1-4/a^2}$}
	\epsfig{file=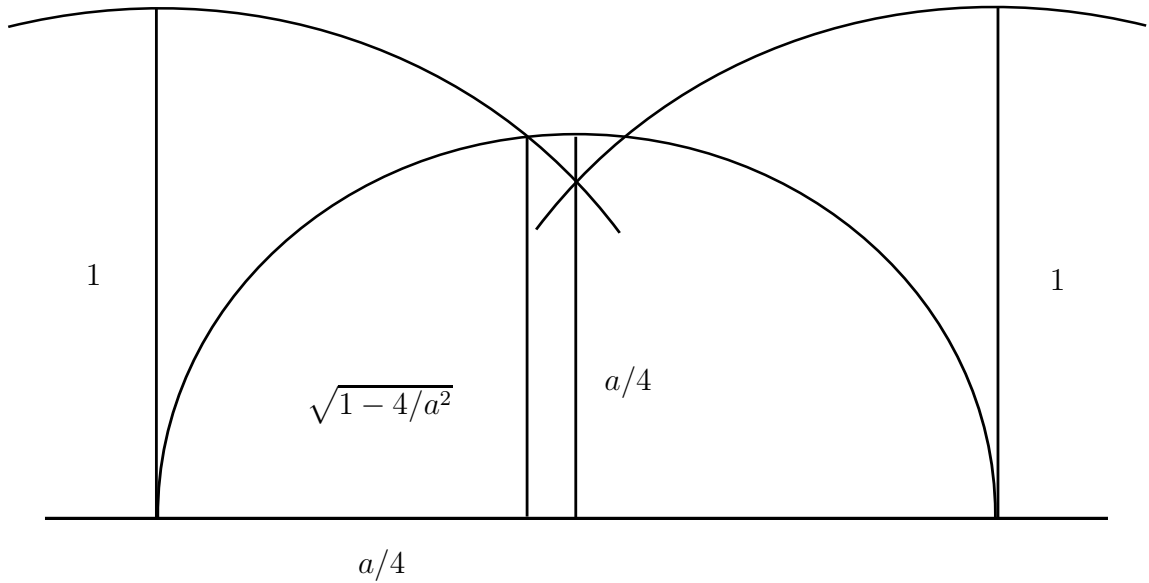, width=\textwidth, height=.5\textwidth}
	\caption{\label{ab=4} The imaginary part of $ku+w_i$ is $>\sqrt{1-4/a^2}$ or $<-\sqrt{1-4/a^2}$ when $a\geq 2\sqrt{2}$}
	\end{center}
\end{figure}
The horoballs $H_{(k-1)u+w_i}$ and $H_{ku+w_i}$ must straddle
the real axis (and therefore the 
strings of $H_{k'a/4}$-balls) for some $k\in \ZZ$, and therefore must
be separated by $2\sqrt{1-4/a^2}$ when $a\geq 2\sqrt{2}$. Since $b=4/a$,
we see that $b=4/a\geq 2 \sqrt{1-4/a^2}$, so  $a\leq 2\sqrt{2}$, thus $a=2\sqrt{2}$. Similarly, we see that
when $a\leq 2\sqrt{2}$, then $b\geq\sqrt{2}$. Exchanging the roles
of $a/2$ and $b$ in the above argument, we see that 
$H_{(k-1)a/2}$ and $H_{ka/2}$ must straddle $u\RR+w_i$ for
some $k$.  As before, we get 
$b\leq \sqrt{2}$ (where we are assuming $b<2$ to get $H_{ka/2}$
disjoint from $H_{(k'+1/2)u+w_i}$).
But this implies in either case that $a=2\sqrt{2}$, $b=\sqrt{2}$,
and $w=\sqrt{2}/2+\sqrt{2}/2i$. Therefore
the horoballs of height $\frac12$ in the two strings
must coincide. This implies
that the arcs $a_{12}$ and $a_{33}$ (the arc
connecting $c_3$ to itself)  in $S$ must
get identified by $f(S)$, which is impossible since
there would be an isometry fixing a lift of $f(a_{12}\cap a_{33})\in \HH^3$. In fact,
this configuration of horoballs does occur in the
Whitehead link complement, but not corresponding
to immersed pants in the manner hypothesized. 

We now turn to the case that $a=4$. Then we have
$b=1$, and this implies that $M$ is the figure eight knot
complement by \cite{A02}. Then $M$ satisfies the conclusion
of the theorem, but we need to show that the immersed
pants in $M$ are of the claimed type, and in fact we show
that this case does not occur. The previous argument
goes through if $f(a_{i3})\neq f(a_{33})$ for $i=1$ or $i=2$, since
this is equivalent to $ka/2+a/4\neq k'u+w_i$, for $k, k'\in \ZZ$. So
assume that $f(a_{13})=f(a_{33})$, and $f(a_{23})=f(a_{33})$,
then we contradict the assumption of the claim. 

Finally, we consider $b=2$, and therefore $a=2$. Again, the
argument above works if $ka/4\neq (k'+1/2)u+w_i$ for all
$k,k'\in \ZZ$, which is equivalent to $f(a_{i3})\neq f(a_{12})$, $i=1, 2$.
So we must have $f(a_{12})=f(a_{13})=f(a_{23})$, contradicting
the assumption of the claim.

{\bf Claim:} {\it $f(a_{13})=f(a_{23})$.}

For contradiction, assume that $f(a_{13})\neq f(a_{23})$. 
Thus, we may assume that $ab<4$. In this case, we
have two full-sized horoballs \cite{A}, and four horoballs
of height $ab/4$ up to the action of the stablizer in $\G$
of $\infty$. We may estimate the ``seen area''
of these horoballs, as in \cite{CM}. The two full-sized 
horoballs each have radius $\frac12$, so they contribute
an area of $\frac{\pi}{2}$. The other four balls may 
be ``overshadowed'' by the full-sized balls, but the
distance of the centers in $\d \HH^3$ must be $\geq \sqrt{ab/4}$
from the centers of the full-sized balls (of radius $\frac12$). Thus, disks
of radius $\sqrt{ab/4}-\frac12$ will be embedded
when centered at the centers of the horoballs. 
Thus we see an area of $\frac{\pi}{2}+ 4\pi ( \sqrt{ab}/2 -\frac12)^{2}$.
But we have $\area(\d \overline{H} )\leq ab$. So
we have 
$$\frac{\pi}{2}+ 4\pi ( \sqrt{ab}/2 -\frac12)^{2} \leq ab.$$
This is a quadratic inequality in $\sqrt{ab}$, and 
we complete the square to see that there are
no positive solutions $\sqrt{ab}$ satisfying the inequality.

This implies that we must have $f(a_{13})=f(a_{23})$. 
In this case, if $a_{i3}$ is oriented away from $c_{3}$, 
then the orientations of $f(a_{13})$ and $f(a_{23})$ must
be reversed under this identification (otherwise the
two arcs would be identified by a parabolic translation,
which is impossible because $f(c_3)$ is embedded). By Proposition \ref{Whiteheadimage}, we see that $M$ is obtained by 
Dehn filling on one component of the Whitehead link
complement. 
\end{proof}

\section{Conclusion}
The result in this paper answers a special case of
the general question as to how singular
can an immersed surface be? One may be able
to extend the results in this paper to understand immersed
twice-punctured two-sided projective planes in non-orientable hyperbolic
3-manifolds. These are totally geodesic for the same
reason that pants are, and it is likely that one could classify
all the non-embedded immersions. It is likely that one could
also give a classification of collections of pants in 
a hyperbolic 3-manifold. That is, to classify the patterns
of intersections that may arise. It's also likely possible
to classify $\pi_1$-injective immersions of punctured tori, and maybe
some other simple surfaces, into hyperbolic 3-manifolds.
It would be interesting to extend the classification of immersed pants to
arbitrary 3-manifolds, by analyzing how the surface
cuts through the various pieces of the geometric 
decomposition. For turnovers immersed in hyperbolic
3-orbifolds,  G. Martin showed that a $(2,3,p)$-triangle
group, $p\geq 7$, in a hyperbolic 3-orbifold is embedded \cite{Martin96}. 
However, there are other triangle groups which are immersed
in 3-orbifolds, see \cite{Maclachlan96, Rafalski07}.

For arbitrary surfaces, it's hard to imagine
a complete classification of immersions into 3-manifolds. We 
conjecture a structural
result  for immersed surfaces:
for a fixed topological type of surface, there are finitely many homeomorphism types of
3-manifolds and $\pi_1$-injective maps of the
surface into these manifolds, such that any $\pi_1$-injective 
immersion of the surface into a 3-manifold factors through an embedding of 
one of these manifolds. This conjecture seems feasible at least
when the target manifold is hyperbolic, and the main result
in this paper proves this conjecture for immersions of pants. 
Rafalski has shown the analogue of this conjecture for 
turnovers immersed in hyperbolic 3-orbifolds \cite{Rafalski07}.

\bibliographystyle{abbrv}
\bibliography{pants.bbl}

\def\cprime{$'$} \def\cprime{$'$} \def\cprime{$'$}
\begin{thebibliography}{10}

\bibitem{Adams85}
C.~C. Adams.
\newblock Thrice-punctured spheres in hyperbolic {$3$}-manifolds.
\newblock {\em Trans. Amer. Math. Soc.}, 287(2):645--656, 1985.

\bibitem{A}
C.~C. Adams.
\newblock The noncompact hyperbolic 3-manifold of minimal volume.
\newblock {\em Proceedings of the American Mathematical Society},
  100(4):601--606, 1987.

\bibitem{A02}
C.~C. Adams.
\newblock Waist size for cusps in hyperbolic 3-manifolds.
\newblock {\em Topology}, 41(2):257--270, 2002.

\bibitem{CM}
C.~Cao and G.~R. Meyerhoff.
\newblock The orientable cusped hyperbolic $3$-manifolds of minimum volume.
\newblock {\em Invent. Math.}, 146(3):451--478, 2001.

\bibitem{EM87}
D.~B.~A. Epstein and A.~Marden.
\newblock Convex hulls in hyperbolic space, a theorem of {S}ullivan, and
  measured pleated surfaces.
\newblock In {\em Analytical and geometric aspects of hyperbolic space
  (Coventry/Durham, 1984)}, volume 111 of {\em London Math. Soc. Lecture Note
  Ser.}, pages 113--253. Cambridge Univ. Press, Cambridge, 1987.

\bibitem{FrancavigliaKlaff06}
S.~Francaviglia and B.~Klaff.
\newblock Maximal volume representations are {F}uchsian.
\newblock {\em Geom. Dedicata}, 117:111--124, 2006.

\bibitem{Ga2}
D.~Gabai.
\newblock Foliations and the topology of $3$-manifolds.
\newblock {\em J. Differential Geom.}, 18(3):445--503, 1983.

\bibitem{CH03}
C.~D. Hodgson and S.~P. Kerckhoff.
\newblock Harmonic deformations of hyperbolic 3-manifolds.
\newblock In {\em Kleinian groups and hyperbolic 3-manifolds (Warwick, 2001)},
  volume 299 of {\em London Math. Soc. Lecture Note Ser.}, pages 41--73.
  Cambridge Univ. Press, Cambridge, 2003.

\bibitem{Maclachlan96}
C.~Maclachlan.
\newblock Triangle subgroups of hyperbolic tetrahedral groups.
\newblock {\em Pacific J. Math.}, 176(1):195--203, 1996.

\bibitem{Martin96}
G.~J. Martin.
\newblock Triangle subgroups of {K}leinian groups.
\newblock {\em Comment. Math. Helv.}, 71(3):339--361, 1996.

\bibitem{NR92}
W.~D. Neumann and A.~W. Reid.
\newblock Arithmetic of hyperbolic manifolds.
\newblock In {\em Topology '90 (Columbus, OH, 1990)}, volume~1 of {\em Ohio
  State Univ. Math. Res. Inst. Publ.}, pages 273--310. de Gruyter, Berlin,
  1992.

\bibitem{Rafalski07}
S.~Rafalski.
\newblock Immersed turnovers in hyperbolic 3-orbifolds.
\newblock preprint arXiv:0708.3415, 2007.

\bibitem{Th3}
W.~P. Thurston.
\newblock A norm for the homology of $3$-manifolds.
\newblock {\em Mem. Amer. Math. Soc.}, 59(339):i--vi and 99--130, 1986.

\end{thebibliography}

\end{document}